\documentclass{amsart}[35pt]
\usepackage{amsmath, amsthm, amscd, amssymb, amsfonts, inputenc, tikz}
\UseRawInputEncoding
\usetikzlibrary{shapes,arrows}

\usepackage[margin=3cm]{geometry}

\newtheorem{theorem}{Theorem}[section]
\newtheorem{lem}[theorem]{Lemma}

\newtheorem{cor}[theorem]{Corollary}
\theoremstyle{definition}

\newtheorem{ex}[theorem]{Example}

\theoremstyle{remark}
\newtheorem{rem}[theorem]{Remark}
\numberwithin{equation}{section}
\begin{document}
\title[On the functional equation $f(x+1)=g(x)f(x)$
]{Remarks on the functional equation $f(x+1)=g(x)f(x)$ and a uniqueness theorem for the gamma function}
\author{M.H.Hooshmand}
\address{
Department of Mathematics, Shiraz Branch, Islamic Azad University, Shiraz, Iran}

\email{\tt hadi.hooshmand@gmail.com, hooshmand@iaushiraz.ac.ir}

\subjclass[2000]{40A30,39A10,39B22,26A99}

\keywords{Gamma-type function, asymptotic condition, Boher-Mollerup theorem, Limit summable function
 \indent }
\date{}

----------------------------------------------
\begin{abstract}
The topic of gamma type functions and related functional equation $f(x+1)=g(x)f(x)$ has been seriously studied
from the first half of the twentieth century till now.
Regarding unique solutions of the equation the asymptotic condition
$\displaystyle{\lim_{x \to \infty}}\frac{g(x+w)}{g(x)}=1$, for each $w>0$, is considered in many serial papers
including R. Webster's article (1997). Motivated by
the topic of limit summability of real functions (introduced by the author in 2001), we first show that the asymptotic condition
in the Webster's paper can be reduced to the convergence of the sequence $\frac{g(n+1)}{g(n)}$ to $1$, and then by removing the initial condition
$f(1)=1$ we generalize it. On the other hand,
Rassias and Trif
proved that if $g$ satisfies another appropriate asymptotic condition, then the equation
admits at most one solution $f$, which is eventually $\log$-convex of the second order. We also sow that without changing
the asymptotic condition for this case, a uniqueness theorem similar to the Webster's result
is also valid for eventually $\log$-convex solutions $f$ of the second order.
This result implies a new uniqueness theorem for the gamma function which states
the $\log$-convexity condition in the Bohr-Mollerup Theorem
can be replaced by $\log$-concavity of order two. At last, some important questions about them will be asked.

\end{abstract}
\maketitle
\section{Introduction and Preliminaries}
The gamma function $\Gamma(x)$ was introduced by Euler in 18th century and improved by Legendre, Gauss and Weierstrass.
In 1922, Bohr and Mollerup proved that the Gamma function is unique solution of the equation $f(x+1)=xf(x)$ with $f(1)=1$
($x>0$) if the
$\log$-convexity condition is considered. Regarding generalizing this equation the following functional equations were
studied:
$$
f(x+1)=g(x)f(x)\; , \; f(x+1)=g(x)+f(x);\;\;\; x>0.
$$
In 1949, Krull \cite{K} studied these equations and obtained many important results about them.
In \cite{Wb}, Webster studied the equation $f(x+1)=g(x)f(x)$ and its special solution namely  gamma type function.
In order to introduce ultra exponential and infra logarithm functions,
the author proposed the topic of {\em limit summability of  functions} (in 2001, \cite{MH}) and showed that
the subject ``gamma type functions'' is its sub-topic. In the topic, domain of the function does not need to be $\mathbb{R}^+$ or
$(a,\infty)$, and it is enough to be defined in positive integers.
We can state a summary of limit summability as follows.\\
Let $f$ be a real or complex function  with domain $D_f\supseteq \mathbb{N}^*=\{1,2,3, \cdots \}$. Put
$$\Sigma_f =\{ x|x+\mathbb{N}^* \subseteq D_f\},$$ and then for any $x\in
\Sigma_f$ and $n\in \mathbb{N}^*$ set $$R_n(f,x) \; = \;
R_n(x):=f(n)-f(x+n),$$
$$f_{\sigma _n}(x)=f_{\sigma _{\ell,n}}(x):=xf(n)+\sum ^n _{k=1} R_k(x).$$
The function $f$ is called limit summable
at $x_0\in \Sigma_f$ if the functional sequence $\{f_{\sigma
_n}(x)\}$ is convergent at $x=x_0$. The function $f$ is called
limit summable on the set $S\subseteq \Sigma_f$ if it is limit
summable at all the points of $S$.
\\
Now, put
$$
f _\sigma(x)=f_{\sigma _{\ell}}(x)=\lim_{n\rightarrow\infty}f_{\sigma _n}(x)\;
, \; R(x)=R(f,x)=\lim_{n\rightarrow\infty}R_n(f,x).
$$
Therefore
$D_{f_\sigma}
 = \{x\in \Sigma_f |f \; \mbox {is limit summable at} \; x\}$, and
 $f_{\sigma_\ell} =f_\sigma $ is the same limit function of $f_{\sigma _n}$ with
 domain $D_{f _\sigma}$.\\
The function $f$ is called limit summable if it is summable on
$\Sigma_f$, $R(1)=0$ and $D_f\subseteq D_f-1$. In this case the
function $f_\sigma$ is referred to as the limit summand function
of $f$. If $f$ is limit summable, then $D_{f _\sigma}=D_f-1$   and
$$ f_\sigma
(x)=f(x)+f_\sigma (x-1)\;\; ; \;\; \forall x\in D_f
$$
Therefore, if $f$ is limit summable then its limit summand function $f_\sigma$
satisfies the well-known difference functional equation $\varphi(x)-\varphi(x-1)=f(x)$
(e.g., see [2, 3, 4]). Hence,
$$f_\sigma(m)=f(1)+\cdots +f(m)=\sum_{j=1}^mf(j)\;\; ; \;\; \forall m\in \mathbb{N}^*.$$
If $f$ is limit summable then one may use the notation $\sigma_{\ell} (f(x))$ instead of
$f_{\sigma_\ell}(x)$.\\
Very often if a real function $f$ is limit summable on an interval of length $1$ and the sequence $R_n(f,1)$
is convergent,
then $f$ is limit summable on whole $\Sigma_f$.
One can see some criteria for existence of unique solutions of the above functional equation in \cite{MH}.
For example, if  $0<b\neq 1$ and $0<a<1$, then the real
function $f(x)=a^x+\log _b x$ is limit summable and
$$f_\sigma (x)=\frac{a}{a-1}(a^x-1)+\log _b \Gamma (x+1).$$
Now, let come back to gamma type functions satisfying the equation. The next theorem is one of the
main theorems about it.\\
{\bf Theorem A}([10, R.J. Webster]). Let the function $g:\mathbb{R}^+\rightarrow \mathbb{R}^+$ have the property
\begin{equation}
\displaystyle{\lim_{x \to \infty}}\frac{g(x+w)}{g(x)}=1\; , \;\;\; \mbox{for all } w>0.
\end{equation}
Suppose that
$f:\mathbb{R}^+\rightarrow \mathbb{R}^+$ is an
eventually $\log$-convex function satisfying the functional equation $f (x + 1) = g(x)f (x)$
 for $x>0$ and the initial condition $f(1)=1$. Then $f$ is uniquely
determined by $g$ through the equation
\begin{equation}
f(x)=\lim_{n\rightarrow \infty}\frac{g(n)\cdots g(1)g(n)^x}{g(n+x)\cdots g(x)}\; ; \; x>0.
\end{equation}
In equation $(1.2)$ the function $f$ is referred as the ``gamma type function of $g$''.
\section{Results for $\log$-convex and $\log$-concave solutions}
Now, we prove that the asymptotic condition of Theorem A can be weak such that the same consequent holds.
\begin{lem}
The asymptotic condition in Theorem A can be replaced by $\frac{g(n+1)}{g(n)}\rightarrow 1$ as $n\rightarrow \infty$.
\end{lem}
\begin{proof}
Let $M$ be a number that $f$ is $\log$-convex on $(M,\infty)$ and fix an arbitrary
non-integer number $x>M$ (it is clear for positive integers). With do attention the hypothesis of the
theorem and using the $\log$-convexity of $f$ on
$[n,n+1,n+1+x-\lfloor x\rfloor,n+2]$ where $n$ is every integer with $n>\max\{M,\lfloor x\rfloor\}$,
one can write
\begin{align}
\left( \frac{f(n+1)}{f(n)}\right)^{x-\lfloor x\rfloor}
\leq
\frac{f(n+1+x-\lfloor x\rfloor)}{f(n+1)} \leq
\left( \frac{f(n+2)}{f(n+1)}\right)^{x-\lfloor x\rfloor}
\end{align}
\begin{align}
f(t+n)=g(t+n-1)\cdots g(t)f(t)\; ; \;\forall t>0
\end{align}
\begin{align}
f(n+1)=g(n)\cdots g(1)
\end{align}
Combining $(2.1)$, $(2.2)$ and $(2.3)$ and putting $k=n-\lfloor x\rfloor$ we find that
$$
\left( \frac{g(n)}{g(k)}\right)^{x}\frac{g(n)\cdots g(k+1)}{g(n)\cdots g(n)}\leq
\frac{f(x)g(x+k)\cdots g(x)}{g(k)\cdots g(1)g(k)^x}\leq
\left( \frac{g(n+1)}{g(n)}\right)^{x}\frac{g(n)\cdots g(k+1)}{g(n+1)\cdots g(n+1)}.
$$
Since $\frac{g(n+i)}{g(n+j)}\rightarrow 1$ as $n\rightarrow \infty$, for every integers $i,j$,
by letting $n\rightarrow \infty$ in the above last inequalities we arrive at $(1.1)$.
\end{proof}
In the next corollary, we remove
the initial condition and get a generalization of Theorem A.
\begin{cor} (A generalization of the Webster's theorem)
Let the function $g:\mathbb{R}^+\rightarrow \mathbb{R}^+$ have the property
$\frac{g(n+1)}{g(n)}\rightarrow l$ as $n\rightarrow \infty$. Suppose that
$f:\mathbb{R}^+\rightarrow \mathbb{R}^+$ is a function satisfying the functional equation
$f(x+1)=g(x)f(x)$ for $x>0$. If one of the following conditions hold:\\
(a) $f$ is eventually $\log$-convex and $0<l\leq 1$,\\
(b) $f$ is eventually $\log$-concave and $l\geq 1$,\\
 then
the function $\frac{1}{f(1)}f$ is uniquely determined by $g$ trough
the equation
\begin{equation}
\frac{1}{f(1)}f(x)=l^{\frac{x^2+x}{2}}\lim_{n\rightarrow \infty}
\frac{g(n)\cdots g(1)g(n)^x}{g(n+x)\cdots g(x)}
=\frac{l^{\frac{x^2+x}{2}}}{g(x)}
\lim_{n\rightarrow \infty}  \left( g(n)^x\prod_{k=1}^n\frac{g(k)}{g(x+k)} \right)\;\; ; \;\; x>0.
\end{equation}
Therefore, there exists a constant $c$ such that
\begin{equation}
f(x+1)
=c l^{\frac{x^2+x}{2}}e^{(\log g)_\sigma(x)}\;\; ; \;\; x>0.
\end{equation}
\end{cor}
\begin{proof}
If (a) holds, then putting $c:=f(1)$, $G(x):=l^{-x}g(x)$ and $F(x):=c^{-1}L^{\frac{x-x^2}{2}}f(x)$
we conclude that $\frac{G(n+1)}{G(n)}\rightarrow 1$ as $n\rightarrow \infty$,
$F$ is eventually $\log$-convex  and
$$
F(x+1)=G(x)F(x)\; ; \; x>0\; , \; F(1)=1.
$$
Therefore, Lemma 2.1 implies that
$$
F(x)=\frac{1}{G(x)}
\lim_{n\rightarrow \infty}  \left( G(n)^x\prod_{k=1}^n\frac{G(k)}{G(x+k)} \right)=
\frac{l^x}{g(x)}
\lim_{n\rightarrow \infty}  \left( g(n)^x\prod_{k=1}^n\frac{g(k)}{g(x+k)} \right)
$$
$$
=\frac{l^x}{g(x)}
\lim_{n\rightarrow \infty}  e^{x\log g(n)+\sum ^n _{k=1} \left(\log g(k)-\log g(x+k)\right)}=
\frac{l^x}{g(x)}
\lim_{n\rightarrow \infty}  e^{(\log g)_{\sigma_n}(x)}
.
$$
Hence we arrive at $(2.4)$.
\end{proof}
\begin{ex}
Let $a>0$ be a constant real number and consider the functional equation
$$f (x + 1) = xa^x f (x)\; ; \;\;\; x > 0.$$
If $a>1$,  then all eventually $\log$-convex solutions of the equation are
of the form
$$
f(x)=ca^{\frac{x^2-x}{2}}\Gamma(x),
$$
for all real constants $c$.
\end{ex}
\begin{rem}
In \cite{Wb} Webster denotes the class of all eventually $\log$-concave functions
$g:\mathbb{R}^+\rightarrow \mathbb{R}^+$ with the asymptotic property
$(1.1)$, by $\mathcal{G}$
and shows that if $g\in \mathcal{G}$ then the gamma type functional equation has a unique solution and
$f$ is eventually $\log$-convex. We remark that $\mathcal{G}$ is indeed the class
of all eventually $\log$-concave functions
$g:\mathbb{R}^+\rightarrow \mathbb{R}^+$ with the property $\frac{g(n+1)}{g(n)}\rightarrow 1$
as $n\rightarrow \infty$ (consider the $\log$-concavity of $g$ on
$[m,m+1,x,x+w,n,n+1]$ where $m,n$ are integers such that $m+1<x$ and  $x+w<n$).
Hence, in this case $(1.1)$ and convergence of the sequence are equivalent (in the related theorems).
\end{rem}
\section{Results for $\log$-convex and concave solutions of order two}
Let $I$ be an interval of the real line and $f:I\rightarrow \mathbb{R}$ a function. Define
the divided difference of $f$ at the points $x_0<x_1<\cdots<x_{n+1}$ in $I$ by
$$
[x_0;f]=f(x_0)\; , \; [x_0,x_1,\cdots,x_{k+1};f]=\frac{[x_1,\cdots,x_{k};f]-[x_0,\cdots,x_{k-1};f]}{x_k-x_0}
\; ; \; k\geq 1.
$$
The function $f$ is called convex of order $n$ or $n$-convex if for any system
$x_0<x_1<\cdots<x_{n+1}$ of points in $I$ it holds that $[x_0,x_1,\cdots,x_{n+1};f]\geq 0$, and
it is $n$-concave if $-f$ is $n$-convex.
A positive function $f$ is said to be eventually $\log$-convex of order $n$ or  $\log$-convex of order $n$ from a number on,
if $I$ contains a subinterval that is unbounded above and
on which the restriction of $f$ is convex of order $n$ (analogously for the concave case). If
$f:\mathbb{R}^+\rightarrow \mathbb{R}^+$ is eventually $\log$-convex of order $2$, then there is a number $M$
such that for every $u<v<w<z$ in $(M,+\infty)$ the following inequality holds
\begin{equation}
\left(\frac{f(w)}{f(v)}\right)^{z-u}\leq \left(\frac{f(z)}{f(u)}\right)^{w-v}
\end{equation}
A similar inequality holds for $\log$-concave functions of order $2$. Also, if $f$ is three times
differentiable on $I$, then $f$ is convex (resp. concave) of order two if and only if $f'''(x)\geq 0$
(resp. $f'''(x)\leq 0$) for all $x\in I$. Since
$$
(\log\Gamma(x))''=\sum_{k=0}^\infty\frac{1}{(x+k)^2},
\; (\log\Gamma(x))'''=-2\sum_{k=0}^\infty\frac{1}{(x+k)^3}\; ; \; x>0,
$$
then the gamma function is $\log$-convex  and also 2-$\log$-concave. More
information about $\Gamma(x)$ and $n$-convex functions can be seen in \cite{Mat, M, RW, Ro}.
\par
In \cite{TR}, Rassias and Trif use the $\log$-convexity of order two and replace the asymptotic
condition $(1.1)$ by
\begin{equation}
\displaystyle{\lim_{x \to \infty}}\frac{g(x+r)}{x^rg(x)}=a^r\; , \;\;\; \mbox{for some $a>0$ and all } r>0.
\end{equation}
for uniqueness conditions
to the solutions of the Gamma-type functional equation. Of course, their theorem gives a unique solution $f$ rather
than the gamma type function of $g$, and it is usable for some equations such as  $f(x+1)=\Gamma(x)f(x)$ but not $f(x+1)=xf(x)$.
Now, by using the techniques from \cite{MH}, we prove that the same condition $(1.1)$ can be used alongside
$\log$-convexity of order two for uniqueness conditions on the equation, and then we obtain a new
uniqueness theorem for the gamma function.
\begin{theorem}
Let $g:\mathbb{R}^+\rightarrow \mathbb{R}^+$ be a given function satisfying $(1.1)$. If
$f:\mathbb{R}^+\rightarrow \mathbb{R}^+$ is an
eventually $\log$-convex $($or $\log$-concave$)$ function of order two
satisfying the gamma type functional equation $f (x + 1) = g(x)f (x)$
, for $x>0$, with the initial condition $f(1)=1$, then $f$ is uniquely
determined by $g$ through the equation $(1.2)$.
\end{theorem}
\begin{proof}
Let $M$ be a number that $f$ is $\log$-convex of order two on $(M,\infty)$ and fix an arbitrary
$x>M$. Then, for every integer $n>M$ the property $[n,n+1,n+1+x,n+2+\lfloor x\rfloor; \log f]\geq 0$
implies
$$
\left( \frac{f(n)}{f(n+2+\lfloor x\rfloor)}\right)^{x}
\leq
\left( \frac{f(n+1)}{f(n+1+x)}\right)^{2+\lfloor x\rfloor}
=\left( \frac{1}{f(x)}\cdot \frac{g(n)\cdots g(1)}{g(n+x)\cdots g(x)}\right)^{2+\lfloor x\rfloor}.
$$
Thus
$$
f(x)\leq \frac{g(n)\cdots g(1)}{g(n+x)\cdots g(x)}\left(g(n+1+\lfloor x\rfloor)\cdots g(n)\right)
^{\frac{x}{2+\lfloor x\rfloor}}
$$
$$
=\frac{g(n)\cdots g(1)g(n)^x}{g(n+x)\cdots g(x)}
\left(\frac{g(n+1+\lfloor x\rfloor)}{g(n)}\cdots \frac{g(n+1)}{g(n)}\cdot\frac{g(n)}{g(n)}\right)
^{\frac{x}{2+\lfloor x\rfloor}}.
$$
Also, by applying $[n+\lfloor x\rfloor,n+x,n+1+\lfloor x\rfloor,n+2+\lfloor x\rfloor; \log f]\geq 0$
for the case $x$ is non-integer, we conclude that
$$
\left( \frac{f(n+\lfloor x\rfloor)}{f(n+2+\lfloor x\rfloor)}\right)^{1-\{x\}}
\leq
\left( \frac{f(n+x)}{f(n+\lfloor x\rfloor+1)}\right)^{2}
=\left(\frac{g(n+x-1)\cdots g(x)f(x)}{f(n+\lfloor x\rfloor+1)}\right)^{2},
$$
hence
$$
f(x)\geq \frac{g(n)\cdots g(1)g(n)^x}{g(n+x)g(n+x-1)\cdots g(x)}\cdot
\frac{g(n+\lfloor x\rfloor)\cdots g(n+1)g(n+x)}{g(n)^x}
\left(\frac{1}{g(n+\lfloor x\rfloor+1)g(n+\lfloor x\rfloor)}\right)^{\frac{1-\{x\}}{2}}
$$
$$
=\frac{g(n)\cdots g(1)g(n)^x}{g(n+x)\cdots g(x)}\cdot
\frac{g(n+\lfloor x\rfloor)}{g(n)}\cdots \frac{g(n+1)}{g(n)}
\left( \frac{g(n+x)}{g(n)}\right)^{\{x\}}
\left( \frac{g(n+x)}{g(n+1+\lfloor x\rfloor)}\cdot \frac{g(n+x)}{g(n+\lfloor x\rfloor)}\right)^{\frac{1-\{x\}}{2}}.
$$
Therefore, the equation $(1.2)$ holds. Note that if $f$ is $\log$-concave of order two, then considering
the equation $f (x + 1)^{-1} = g(x)^{-1}f (x)^{-1}$ and the above argument, one can get the result.
\end{proof}

\begin{rem}
In view of the above proof, one find that if the sequence $\frac{g(n+1)}{g(n)}$ converges to $1$
then
$$
f(x)\leq \liminf_{n\to\infty}\frac{g(n)\cdots g(1)g(n)^x}{g(n+x)\cdots g(x)}\; ; \; x>0,
$$
and the equality holds if $(1.1)$ is satisfied. Hence, we remains the following
question.\\
{\bf Question I.} If the asymptotic condition $(1.1)$ in Theorem 3.1 is replaced by $\frac{g(n+1)}{g(n)}\to 1$ as $n\rightarrow \infty$,
then does the equality $(1.2)$ hold again?
\end{rem}
\subsection{A new uniqueness theorem for the gamma function}
It is easy to see that the difference functional equation $f(x+1)=xf(x)$ for $x>0$ (for the case $g(x)=x$) , with the initial condition  $f(1)=1$
has infinitely many solutions, having the property $f(n)=(n-1)!$, for every positive integer $n$. In
1922, Bohr and Mollerup proved that the Gamma function $\Gamma(x)$ is its unique solution if the
$\log$-convexity condition is considered. Thereafter, several uniqueness conditions were proved by using monotonicity, geometrically
convexity and so on (e.g., see \cite{A, G, M, RW}). In \cite{Mat} (2015), Matkowski proved some other uniqueness theorems
for it by using the Jensen convexity and some other conditions.
Now, as an interesting result of Theorem 3.1 we conclude that the condition $\log$-convexity in the Bohr-Mollerup Theorem
can be replaced by $\log$-concavity of order two.
\begin{cor}
The gamma function $\Gamma(x)$ is the only function $f$  that has the three properties\\
(a) $f (x + 1) = x f (x)$ for $x > 0$;\\
(b) $f (1) = 1$;\\
(c) $f$  is $\log$-concave of the second order.
\end{cor}
\begin{rem}
Since the Gamma function is $n$-$\log$-convex (resp. $n$-$\log$-concave) for every odd (resp. even) positive integer
$n$, then we can replace the third condition (c) in the Bohr-Mollerup  Theorem  by each of the following conditions:\\
$(c_1)$  $f$  is $n$-$\log$-convex for \underline{every} odd  positive integers $n$,\\
$(c_2)$  $f$  is $n$-$\log$-concave for \underline{every} even  positive integers $n$.\\
Hence the next question arises:\\
{\bf Question II.} Can  the third condition (c) be replaced by one of the next conditions?:\\
 $(c_3)$  $f$  is $n$-$\log$-convex for \underline{some} odd  positive integer $n$,\\
$(c_4)$  $f$  is $n$-$\log$-concave for \underline{some} even  positive integer $n$,\\
 $(c_5)$  $f$  is $n$-$\log$-convex and $m$-$\log$-concave  for some odd  positive integer $n$
 and even positive integer $m$,\\
 $(c_6)$  $f$  is $n$-$\log$-convex and $(n+1)$-$\log$-concave  for some odd  positive integer $n$.
\end{rem}

\end{document}